\documentclass{amsart}

\usepackage{mathrsfs}

\theoremstyle{plain}
\newtheorem{theorem}{Theorem}
\newtheorem{proposition}{Proposition}
\newtheorem{corollary}{Corollary}
\newtheorem{lemma}[theorem]{Lemma}

\theoremstyle{definition}
\newtheorem{definition}{Definition}
\newtheorem{example}{Example}

\theoremstyle{remark}
\newtheorem{remark}{Remark}

\begin{document}

\title[Anti-holomorphic involutions and spherical subgroups]
{Anti-holomorphic involutions and spherical subgroups of reductive groups}

\author{St\'ephanie Cupit-Foutou}
\thanks{This research was funded  by the SFB/TR 12 of the German Research Foundation (DFG) and partially by the DFG priority program SPP 1388-Darstellungstheorie}
\email {Stephanie.Cupit@rub.de}


\maketitle

\begin{abstract}
We study the action of an anti-holomorphic involution $\sigma$ of a connected reductive complex algebraic group $G$ on the set of spherical subgroups of $G$.
The results are applied to $\sigma$-equivariant real structures on spherical homogeneous $G$-spaces admitting a wonderful embedding.
Using combinatorial invariants of these varieties, we give an existence and uniqueness criterion for such real structures.
We also investigate the associated real parts of the wonderful varieties.
\end{abstract}

\section*{Introduction}

The classification of anti-holomorphic involutions of connected reductive complex algebraic groups $G$ is well-established.
In this article, we study the behavior of certain subgroups of $G$ under these mappings.
More precisely, we are interested in the subgroups of $G$ called spherical -
groups including e.g. parabolic and symmetric subgroups of $G$.
The fixed point sets of these subgroups w.r.t. such involutions naturally provide examples of real spherical spaces.

Spherical subgroups of a given connected reductive group can be classified by means of combinatorial objects (see~\cite{Lu01,Lo,BP,CF}).
In terms of these invariants, we establish
existence and uniqueness criteria about anti-holomorphic involutions on spherical homogeneous spaces, extending results obtained in~\cite{ACF,Ak2}, 
in particular by considering a wider class of homogeneous $G$-spaces.

To any spherical subgroup of $G$, one can assign its so-called spherical closure, a spherical subgroup of $G$ canonically defined and whose associated homogeneous
space has a wonderful embedding.
As proved by Luna in~\cite{Lu01}, spherically closed subgroups of $G$ play a central role in the classification of all spherical subgroups of $G$.
Moreover, for such subgroups, the aforementioned combinatorial classification is really convenient since the involved objects are built on the Dynkin diagram of $G$.
After having proved that this class of subgroups $H$ of $G$ is preserved by any anti-holomorphic involution $\sigma$ of $G$, we study qualitative and quantitative
properties of so-called $\sigma$-equivariant real structures on $G/H$ and their wonderful compactifications.

Like in~\cite{Ak2}, a peculiar automorphism of Dynkin diagram plays a central role in the present work; we recall its definition and properties in Section 1 as well
as the notion of Cartan index that will be needed for quantitative results.

Once the basic material concerning spherical subgroups of $G$ and their invariants is set up,
we study, in Section 2, properties of $\sigma$-conjugates of spherical subgroups of $G$.
Theorem~\ref{conjugationcriterion} enables us to decide whether a spherical subgroup $H$ of $G$ is conjugate to $\sigma(H)$;
if the latter happens to hold, we are able to prove
an existence and uniqueness statement concerning $\sigma$-equivariant real structures on $G/H$; see Corollary~\ref{uniquenessspherclosed}.

In Section 3, we investigate how $\sigma$-equivariant real structures are carried over through geometrical operations like Cartesian products, parabolic inductions.
 
In the last section, we focus on the so-called wonderful compactifications $X$ of spherically closed homogeneous spaces.
Theorem~\ref{criterionwithsphericalsystem} provides a criterion for a $\sigma$-equivariant real structure to exist on $X$.
As an application of this result, we obtain, in particular, that (almost) all primitive self-normalizing spherical subgroups $H$ of $G$ 
are conjugate to $\sigma(H)$; see Theorem~\ref{conjugateprimitive} and Remark~\ref{remark-strictw'ful}.
Finally, we study the real parts of wonderful varieties equipped with a $\sigma$-equivariant real structure;
we give several conditions for the existence of real points (Theorem~\ref{Thm-reallocus} and Proposition~\ref{pptyS})
and we conclude our work by presenting a few examples illustrating how these real loci can be diverse and various.




\bigbreak\noindent
\paragraph{\textit{Acknowledgment.}}
I am very grateful to D. N. Akhiezer for stimulating and helpful discussions.
This work was accomplished while D. N. Akhiezer, supported by the SPP-1388 Darstellungstheorie, was visiting Bochum Ruhr-Universit\"at in Winter 2014.
I thank also the referees for their comments which helped me to improve the organization of this paper.

\subsection{Notation and terminology}
Let $G$ be a complex semisimple group and $\sigma$ be an anti-holomorphic involution of $G$.

We fix a Borel subgroup $B$ of $G$ and we choose a maximal torus $T\subset B$ stable by $\sigma$.
We denote the related set of simple roots by $S$ and the Weyl group $N_G(T)/T$ by $W$.
Let $\mathcal X(T)$ be the character group of $T$. Then $\sigma$ defines an automorphism $\sigma^\top$ on $\mathcal{X}(T)$
while setting
$$
\sigma^\top(\chi)=\overline{\chi\circ\sigma}\qquad\mbox{ for $\chi \in\mathcal X(T)$.}
$$
Given a representation $(V,\rho)$ of $G$, we denote the corresponding $\sigma$-twisted module by $V^\sigma$.
Specifically, $V^\sigma$ is given by the complex conjugate $\overline{V}$ of $V$ equipped with the $G$-module structure
$$
g\mapsto\overline{\rho(\sigma(g))}\qquad\mbox{for any $g\in G$}.
$$

In this paper, we are concerned with anti-holomorphic involutions (also called real structures) on algebraic varieties.
Given a real structure $\mu$ on an algebraic variety $X$, its corresponding real part $X^\mu$ is defined as follows:
$$
X^{\mu}=\{x\in X: \mu(x)=x\}.
$$

Let $X$ be a $G$-variety equipped with a real structure $\mu$.
The mapping $\mu$ is called \emph{$\sigma$-equivariant} if
$$
\mu(gx)=\sigma(g)\mu(x)\quad\mbox{ for all $(g,x)\in G\times X$}.
$$

\section{Dynkin diagram automorphism associated to $\sigma$}\label{recallsauto}
To the involution $\sigma$, we can attach a partition $S=S_0\cup S_1$ as well as  an involutive map $\omega:S_1\rightarrow S_1$.
Specifically, the elements of $S_0$ (resp. $S_1$) correspond to the black (resp. white) circles of the Satake diagram associated to $\sigma$
whereas  $\alpha$ and $\omega(\alpha)$ for $\alpha\in S_0$ are the vertices of a bi-oriented edge of this diagram.  

Consider the subgroup of $W$ generated by the simple reflections associated to the elements of $S_0$
and let $w_\bullet$ denote its element of maximal length.
Following~\cite{Ak}, we set
$$
\varepsilon_\sigma( \alpha)=\left\{
	\begin{array}{ll}
		\omega(\alpha)  & \mbox{if } \alpha\in S_1 \\
		-w_\bullet(\alpha) & \mbox{if } \alpha\in S_0
	\end{array}.
\right.
$$

\begin{theorem}[\cite{Ak}]\label{autoAkhiezer}
\smallbreak\noindent
{\rm (i)}\enspace 
The map $\varepsilon_\sigma$ is an automorphism of $S$.
Further, it is induced by a self-map on $\mathcal X(T)$ (still denoted by $\varepsilon_\sigma$).
\smallbreak\noindent
{\rm (ii)}\enspace
If $n_\bullet\in N_G(T)$ represents $w_\bullet$ then $n_\bullet\sigma(B)n_\bullet^{-1} =B$.
\smallbreak\noindent
{\rm (iii)}\enspace
If $V$ is a simple $G$-module of highest weight $\lambda$ then  $V^\sigma$ is also a simple $G$-module; its highest weight equals $\varepsilon_\sigma(\lambda)$.
Further,
$$
\varepsilon_\sigma(\lambda)=w_\bullet(\sigma^\top(\lambda)).
$$
\end{theorem}

\begin{proof}
See more precisely Theorem 3.1 and its proof in loc. cit..
\end{proof}

Let $V$ be a $\sigma$-self-conjugate simple $G$-module $V$ of highest weight $\lambda$ that is, $\varepsilon_\sigma(\lambda)=\lambda$.
Then there exists a anti-linear automorphism $\nu: V\rightarrow V$.
Further, $\nu^2=c\, \mathrm{Id}_V$ with $c\in\mathbb R^\times$ and the sign of $c$ does not depend on $\nu$; see e.g. Proposition 8.2 in~\cite{O}.
The sign of $c$ is called the Cartan index of $V$. 

\section{$\sigma$-conjugates of spherical subgroups of $G$}

First, let us recall that a normal $G$-variety is said to be \emph{spherical} if it has an open $B$-orbit.
Note that the notion of sphericity makes sense also when $G$ is connected and  reductive (non necessarily semisimple).
Analogously, a subgroup $H$ of $G$ is called spherical if the $G$-variety $G/H$ is spherical.

Parallel to the complex case, there is the notion of a real spherical variety for real semisimple Lie groups $G_{\mathbb R}$.
A normal real algebraic $G_{\mathbb R}$-variety is called \emph{real spherical} when a minimal parabolic subgroup of $G_{\mathbb R}$ has an open orbit on it.
Whenever it is not empty, the real part of a spherical $G$-variety w.r.t. a $\sigma$-equivariant real structure is an example of a real spherical $G^\sigma$-variety.

In the following, $H$ denotes a spherical subgroup of $G$. Without loss of generality, we assume that $BH$ is open in $G$.

Set
\begin{equation}
\mu_\sigma:	G/H\longrightarrow G/\sigma(H), \quad
 gH \longmapsto \sigma(g)\sigma(H).
\end{equation}

Note that $G/\sigma(H)$ is spherical since so is $G/H$.

\subsection{}\label{anysphericalgp}
Following~\cite{LV}, to $G/H$ we attach three combinatorial invariants (so called Luna-Vust invariants):
its set of colors $\mathcal D=\mathcal D (G/H)$, its weight lattice $\mathcal{X}=\mathcal{X}(G/H)$ and its valuation cone $\mathcal V=\mathcal V(G/H)$.
The set of  colors $\mathcal D=\mathcal D (G/H)$ is the set of $B$-stable prime divisors of $G/H$;
the lattice $\mathcal{X}$ consists of the $B$-weights of the function field $\mathbb C(G/H)$ of $G/H$;
the valuation cone $\mathcal V$ is the set of $G$-invariant $\mathbb Q$-valued valuations of $\mathbb C(G/H)$.

Any valuation $v$ defines a homomorphism
$$
 \mathbb C(G/H)\longrightarrow\mathbb Q,\quad \rho: f\longmapsto v(f)
$$
and in turn $v$ induces an element $\rho_v$ of $V:=\mathrm{Hom}(\mathcal X(G/H),\mathbb Q)$; see loc. cit. for details.
This yields in particular two maps: 
$$
\mathcal V\longrightarrow V, \quad v\longmapsto \rho_v \quad\mbox{ and } \qquad \mathcal D\longrightarrow V, \quad D\longmapsto \rho_D
$$
where $\rho_D:=\rho_{v_D}$ and $v_D$ is the valuation of the divisor $D$.

The first map happens to be injective hence we can regard $\mathcal{V}$ as a subset of $V$.
The second map may not be injective; the set $\mathcal D$ is thus equipped with the map $\mathcal D\rightarrow V$ together with an additional map
$D\mapsto G_D$ with $G_D\subset G$ being the stabilizer of the color $D$.
By $\mathcal D(G/H_1)=\mathcal D(G/H_2)$, we  just mean that there is a bijection $\varphi:\mathcal D(G/H_1)\rightarrow\mathcal D(G/H_2)$
such that $\rho_D=\rho_{\varphi(D)}$ and $G_D=G_{\varphi(D)}$ for every $D\in\mathcal D(G/H_1)$.

A spherical homogeneous space is uniquely determined (up to $G$-isomorphism) by its Luna-Vust invariants; see  Losev's results in~\cite{Lo}.
In case $G/H$ is affine, these three invariants can be replaced by a single one: the weight lattice $\Gamma=\Gamma(G/H)$, that is  the set given by the highest weights 
of the coordinate ring of $G/H$ considered as a $G$-module; see again~\cite{Lo}.

\subsection{} Thanks to~\cite{Ak2}\footnote{See also~\cite{Hu} for analogy.} the relations between the Luna-Vust invariants of $G/H$ and those of $G/\sigma(H)$ are well-understood.
Specifically, we have the following description involving the automorphism $\varepsilon_\sigma$ of $S$ (see Section~\ref{recallsauto} for recollection of its definition). 

\begin{lemma} \label{Akhiezerlemma}
If $H$ is a spherical subgroup of $G$ then 
\begin{enumerate}
\item $\mathcal X\left(G/\sigma(H)\right)=\varepsilon_\sigma(\mathcal{X})$,
 \item $\mathcal{V}(G/\sigma(H))=\varepsilon_\sigma(\mathcal{V})$ and
\item $\mathcal D\left(G/\sigma(H)\right)=\left\{\mu_\sigma(n_\bullet D): D\in\mathcal{D}\right\}$
equipped with the maps 
$$
\mu_\sigma(n_\bullet D)\mapsto\varepsilon_\sigma(\rho_D) \quad \mbox{ and }\quad
\mu_\sigma(n_\bullet D)\mapsto n_\bullet\sigma(G_D)n_\bullet^{-1}
$$
with $n_\bullet$ being a representative in $N_G(T)$ of $w_\bullet$.
\item
If $H$ is also reductive then $\Gamma\left(G/\sigma(H)\right)=\varepsilon_\sigma(\Gamma)$.
\end{enumerate}

\end{lemma}

\begin{proof}
The three first assertions stem from Proposition 5.2, Proposition 5.3 and Proposition 5.4 in~\cite{Ak2} resp. 
whereas the fourth one follows from Theorem 2.1 in~\cite{Ak} 
(see also the first paragraph of the proof of Theorem 6.1 therein).
\end{proof}

\begin{definition}
The set $\mathcal D$ of colors of a spherical homogeneous space $G/H$ is called \emph{$\varepsilon_\sigma$-stable} if
for every $D\in\mathcal{D}$, there exists $D'\in\mathcal{D}$ (depending on $D$) such that 
$$
\varepsilon_\sigma(\rho_D)=\rho_{D'}\quad \mbox{ and } \quad n_\bullet\sigma(G_D)n_\bullet^{-1}=G_{D'}.
$$

\end{definition}

\begin{theorem}\label{conjugationcriterion}
Let $H$ be a spherical subgroup of $G$.
\begin{enumerate}
 \item 
The subgroups $H$ and $\sigma(H)$ of $G$ are conjugate if and only if 
the combinatorial invariants of $G/H$ are $\varepsilon_\sigma$-stable.
\item 
In case $G/H$ is affine, the subgroups $H$ and $\sigma(H)$ of $G$ are conjugate if and only if  the weight monoid of $G/H$ is $\varepsilon_\sigma$-stable.
\end{enumerate}
\end{theorem}

\begin{proof}
Thanks to the aforementioned Losev's results, it suffices to prove that $\mathcal X=\mathcal X(G/\sigma(H))$, $\mathcal V=\mathcal V(G/\sigma(H))$ and
$\mathcal D=\mathcal{D}(G/\sigma(H))$ (in the sense recalled above).
The required equalities are thus given by our assumption of $\varepsilon_\sigma$-stability together with Lemma~\ref{Akhiezerlemma}. 
\end{proof}

\begin{example}
Let $H=B^-$ with $B^-$ being the Borel subgroup of $G$  opposite to $B$.
First, recall that since $T$ is $\sigma$-stable, $\sigma(B^-)$ is conjugate to $B^-$.
Secondly, $\mathcal V=\mathcal X=\{0\}$ and $\mathcal D=\{B s_\alpha B^-/B^-:\alpha\in S \}$ with $s_\alpha\in W$ being the simple reflection associated to $\alpha$.
These invariants are clearly $\varepsilon_\sigma$-stable; note that $\mathcal D$ may not be fixed by $\varepsilon_\sigma$ (for instance in the case of the quasi-split but non-split real form 
in type $\mathsf A$).
\end{example}

\begin{remark}
In Proposition 5.4 in~\cite{Ak2}, the $\varepsilon_\sigma$-stability assumption on $\mathcal D$ is replaced by the stronger conditions:
$$
\varepsilon_\sigma(\rho_D)=\rho_{D}\quad \mbox{ and } \quad n_\bullet\sigma(G_D)n_\bullet^{-1}=G_{D}.
$$
These condition leave aside many cases as the preceding example shows.
\end{remark}

We now turn to some peculiar classes of spherical groups, namely, to wonderful and spherically closed subgroups of $G$.
\subsection{}

Equivariant embeddings of $G/H$ are classified by a finite family of couples $(\mathcal C,\mathcal F)$, subject to restrictions, with $\mathcal C$
being a finitely generated strictly convex cone in $V$ and $\mathcal F$ being a subset of $\mathcal D$. 
In case $\mathcal V$ is strictly convex (equivalently, if $N_G(H)/H$ is finite), the couple $(\mathcal V,\emptyset)$ is admissible
and thus corresponds to an equivariant embedding of $G/H$.
This embedding is complete; it is called \emph{the canonical embedding of $G/H$}.

If $N_G(H)/H$ is finite and its canonical embedding is smooth, the spherical subgroup $H\subset G$ is called \emph{wonderful}. 
Spherically closed subgroups of $G$ (see below for the definition)
and normalizers of spherical subgroups of $G$ are examples of wonderful subgroups of $G$; see~\cite{K}.

\begin{proposition}\label{wonderfulness}
Let $H$ be a spherical subgroup of $G$.
Then $H\subset G$ is wonderful if and only if $\sigma(H)$ is a wonderful subgroup of $G$.
\end{proposition}

\begin{proof}
First remark that $\sigma(H)\subset G$ is spherical.
By~Lemma~\ref{Akhiezerlemma}, $\mathcal V\left(G/\sigma(H)\right)=\varepsilon_\sigma\left(\mathcal V(G/H)\right)$ and 
$V\left(G/\sigma(H)\right)=\varepsilon_\sigma\left(V(G/H)\right)$.
It follows that the valuation cone of $G/\sigma(H)$ is strictly convex if so is $\mathcal V(G/H)$ and vice versa.
Finally, the canonical embedding (whenever it exists) of any spherical $G/H'$ is smooth if and only if
$\mathcal V(G/H')$ is generated by a basis of $V(G/H')$; see Section 4 in~\cite{B2}.
This criterion allows to conclude the proof.
\end{proof}

\subsection{}\label{recall-color}
For later purposes, we shall need the following properties of wonderful subgroups of $G$.

Consider the natural epimorphism 
$$\pi: G\rightarrow G/H
$$
with $H\subset G$ being spherical.

Given $D\in\mathcal D$, $\pi^{-1}(D)$ is a $B$-stable prime divisor of $G$.
Replacing $G$ by a finite covering, we can (and do) assume now that $G$ is simply connected.
Then there exists a unique $(B\times H)$-eigenfunction $f_D$ in $\mathbb C[G]$ defining $\pi^{-1}(D)$ and such that $f_D$ equals $1$ on the neutral element of $G$.
We denote the $(B\times H)$-weight of $f_D$ by $(\omega_D,\chi_D)$; the restrictions of $\omega_D$ and $\chi_D$ onto $B\cap H$ coincide.

Let $\mathcal X(B)$ (resp. $\mathcal X(H)$) be the set of characters of $B$ (resp. $H$).
\begin{lemma}\label{weight-colors}
If $H\subset G$ is wonderful then the following assertions hold.
\begin{enumerate}
 \item The $(B\times H)$-eigenfunctions  of $\mathbb C[G]$ are the monomials in the $f_D$'s, $D\in\mathcal D$.
\item The abelian group $\mathcal X(B)\times _{\mathcal X(B\cap H)}\mathcal X (H)$ is isomorphic to the Picard group of X.
Moreover, it is freely generated by the couples $(\omega_D,\chi_D), D\in\mathcal D$.

\end{enumerate}

\end{lemma}

\begin{proof}
See Section 6 in~\cite{Lu01} as well as Subsection 3.2 in~\cite{B3}.
\end{proof}

\subsection{}
We consider now spherically closed subgroups of $G$. 

First, let us recall their definition.
By $N_G(H)$ we denote the normalizer of $H$ in $G$.
The group of $G$-automorphisms of $G/H$ can be identified with $N_G(H)/H$ and $N_G(H)$ acts naturally on the set $\mathcal D$ of colors of $G/H$.
The spherical closure $\overline H$ of $H$ is defined as the subgroup of $N_G(H)$ which fixes each element of $\mathcal D$;
the group $\overline{H}$  is a spherical subgroup of $G$ containing $H$.
The group $H$ is called \emph{spherically closed} if $H=\overline H$.

\begin{example}
\begin{enumerate}
\item
Let $G=SL_2$.
The $G$-variety $G/T$ is isomorphic to the complement of the diagonal of $\mathbb P^1\times\mathbb P^1$;
it is spherical and has two colors, both having the same $B$-weight, that is the fundamental weight of $G$.
The normalizer of $T$ in $G$ exchanges these two colors; since $N_G(T)/T\simeq \mathbb Z_2$, the subgroup $T$ of $G$ is spherically closed.
\item
Let $G=SO_{2n+1}$ and $H$ be the stabilizer of a non-isotropic line in $\mathbb C^{2n+1}$.
The subgroup $H\subset G$ is self-normalizing and not connected.
Let $H^\circ$ be the identity component of $H$. 
Then $G/H$ and $G/H^\circ$  are spherical $G$-varieties and they both have only one color. 
The subgroup $H\subset G$ is spherically closed whereas $H^\circ$ is not since $H^\circ\neq \overline{H^\circ}=N_G(H^\circ)=H$.
\end{enumerate}
\end{example}

\begin{proposition}\label{antiholospherclosed}
If $H$ is a spherically closed subgroup of $G$ then so is $\sigma(H)$.
\end{proposition}

\begin{proof}
Let us first remark that the normalizer of $\sigma(H)$ in $G$ equals $\sigma(N_G(H))$.
Consider now an element $n\in\sigma(N_G(H))$ which fixes every $\sigma(B)$-color of $G/\sigma(H)$.
We shall prove that $\sigma(n)\in H$. 

Let $\pi_\sigma: G\rightarrow G/\sigma(H)$ be the natural projection and 
$f_D\in \mathbb C[G]$ define the equation of $\pi_\sigma^{-1}(D)$ with $D$ being a $\sigma(B)$-color of $G/\sigma(H)$.
By assumption on $n$, we have: $n\cdot f_D=f_D$.
It follows that  
$$
\begin{array}{ll}
\left( \sigma(n)\cdot (\overline{f_D\circ \sigma}) \right)(g) &= (\overline{f_D\circ\sigma})(g\sigma(n)) \\
& =\overline{f_D(\sigma(g)n)}=\overline{(n\cdot f_D)(\sigma(g))} \\
& = \overline{f_D(\sigma(g))}  =(\overline{f_D\circ\sigma})(g).
\end{array}
$$

Note that $\overline{f_D\circ \sigma}=f_{\mu_\sigma^{-1}(D)}$ and the set $\mathcal D$ consists of the elements $\mu_\sigma^{-1}(D)$
with $D$ a $\sigma(B)$-color of $G/\sigma(H)$.
We thus obtain that $\sigma(n)\in N_G(H)$ fixes every color of $G/H$ and since $H$ is supposed to be spherically closed, this gives as desired: $\sigma(n)\in H$.

\end{proof}

\begin{corollary}\label{uniquenessspherclosed}
Let $H\subset G$ be spherical and $\sigma(H)=aHa^{-1}$ with $a\in G$.
If further $H\subset G$ is spherically closed then
the map 
$$
\mu: G/H\longrightarrow G/H, \quad gH\longmapsto \sigma(g)a H
$$
is the unique $\sigma$-equivariant real structure on $G/H$.
\end{corollary}

\begin{proof}
The map $\mu^2$ defines a $G$-equivariant automorphism of $G/H$ hence yields a bijection of $\mathcal D$.
Note that $\mu^2(gH)=g\sigma(a) aH$.
Since $\sigma$ is an involution of $G$, we have: $H=\sigma^2(H)=\sigma(a)aHa^{-1}\sigma(a)^{-1}$ and in turn, $\sigma(a)a\in N_G(H)$.

Let $\tilde\mu: G\rightarrow G$ be the map defined by $g\mapsto \sigma(g)a$.
We thus have: $\mu\circ\pi=\pi\circ\tilde\mu$.
Note that the function $f_D$ is mapped through $\tilde \mu^2$ to the complex conjugate of $\overline{f_D\circ\tilde\mu}\circ\tilde\mu$;
the latter is obviously a $(B\times H)$-eigenfunction of weight
$(\omega_D,\chi_D)$-- the weight of $f_D$.
By Lemma~\ref{weight-colors}, $\mu^2$ thus fixes each element of the set $\mathcal D$; 
this implies that $\sigma(a)a$ is an element of the spherical closure of $H$ which is $H$ itself by hypothesis.

To prove the uniqueness assertion, we consider another $\sigma$-equivariant real structure on $G/H$, say $\mu'$.
The map $\mu\circ \mu'$ is thus a $G$-automorphism of $G/H$ hence it is given by an element, say $n$, of $N_G(H)$.
Arguing similarly as we have just done to prove that $\mu$ is involutive, 
we can show that $n\in H$. This yields as desired that $\mu\circ \mu'$ is the identity map on $G/H$.
\end{proof}

\begin{example}
Let $G=SL_2$ and $H=T=\overline{T}$.
The map $\mu$ stated in Corollary~\ref{uniquenessspherclosed} is clearly an involution for every $\sigma$ such that $\sigma(T)=T$. 
\end{example}

\section{Geometric operations on varieties and related constructions of real structures}

We investigate now how real structures and the related real parts (definition recalled right below) are carried over through geometrical operations on varieties:
Cartesian product, parabolic induction.

\subsection{}
Let $X_1$ and $X_2$ be two complex manifolds, equipped with real structures $\mu_1$ and $\mu_2$ respectively.
Then the direct product $(\mu_1,\mu_2)$ is obviously a real structure on $X_1\times X_2$. 
Moreover, if there exists an anti-holomorphic diffeomorphism $\tau:X_1\rightarrow X_2$, then
the map $(x_1,x_2)\mapsto(\tau^{-1}(x_2),\tau(x_1))$ defines a real structure on $X_1\times X_2$. 

\subsection{}\label{sectionparabolicinduction}

Let $P$ be any parabolic subgroup of $G$ and let $P=P^uL$ its Levi decomposition with $L$ being the Levi subgroup of $P$ containing $T$.
Given a $L$-variety $X'$, one considers the fiber product $X:=G\times_P X'$ with $P^u$ acting trivially on $X'$.
The variety $X$ is usually called a parabolic induction of $X'$; it is a $G$-variety with the natural action of $G$.

\begin{remark}
Let $P_1$ and $P_2$ be parabolic subgroups of $G$ such that $P_1\cap P_2$ contains a Levi subgroup $L$ of both $P_1$ and $P_2$.
If there exists $n\in G$ such that $P_2=nP_1 n^{-1}$ then
$[g,x]\mapsto [gn, x]$ defines an isomorphism between the $G$-varieties $G\times_{P_1} X'$ and $G\times_{P_2}X'$.
\end{remark}

\begin{lemma}~\label{parabolicinduction}
Let $P$ be a parabolic subgroup of $G$ such that $\sigma(P)=nPn^{-1}$ with $n\in G$.
Suppose further that the Levi factor $L$ of $P$ containing $T$ is $\sigma$-stable.
Let $X'$ be a $L$-variety equipped with a $\sigma_L$-equivariant anti-holomorphic map $\mu'$
(with $\sigma_L$ being the restriction of $\sigma$ onto $L$).
Then
\begin{equation}~\label{inducedstructure}
G\times_P X'\longrightarrow G\times_P X',\qquad [g,x]\longmapsto [\sigma(g)n,\mu'(x)]
 \end{equation}
defines a $\sigma$-equivariant anti-holomorphic diffeomorphism.
\end{lemma}

\begin{proof}
First note that $\sigma(P)$ is a parabolic subgroup of $G$; it contains the Borel subgroup $\sigma(B)$ of $G$.
Since $\sigma(L)=L$, we can consider the parabolic inductions $G\times_{P} X'$ and $G\times_{\sigma(P)} X'$.
In particular, we let the unipotent radicals $P^ u$ and $\sigma(P)^u$ of $P$ and $\sigma(P)$ resp. act trivially on $X'$.
From $\sigma(P)^u=\sigma(P^u)$, we derive for any $(p=p^ul,x)\in P^uL\times X'$ the following equalities:
$$
\mu'(\sigma(p)x)=\mu'(\sigma(l)x)=l\mu'(x)=p\mu'(x)
$$
As a consequence, the assignment $(g,x)\mapsto (\sigma(g),\mu'(x))$
defines an anti-holomorphic map from $G\times_P X'$ to $G\times_{\sigma(P)}X'$.
Moreover, the subgroups $P$ and $\sigma(P)$ of $G$ being conjugate by assumption,
the $G$-varieties $G\times_P X'$ and $G\times_{\sigma(P)}X'$ are isomorphic; see the remark above. 
The lemma follows.
\end{proof}

\begin{proposition}\label{realpointparabolicinduction}
Let $X'$ and $X=G\times_P X'$ satisfy the properties stated in Lemma~\ref{parabolicinduction}.
Suppose also that $X$ is equipped with the diffeomorphism stated in~(\ref{inducedstructure}).
If $X$ contains fixed points w.r.t this diffeomorphism, so does $G/P$ w.r.t $gP\mapsto \sigma(g)P$.
\end{proposition}

\begin{proof}\label{inducingrealpoint}
By assumption, $\sigma(P)=nPn^{-1}$ for some $n\in G$. 
Let $x=[g,z]\in X=G\times_P X'$ be a fixed point, we thus get: $\sigma(g)n=gp^{-1}$.
This implies that $gP$ is a real point of $G/P$ with respect to the real structure $gP\mapsto \sigma(g)n P$.
The proposition follows.
\end{proof}

\section{Wonderful varieties}\label{wonderful}

We shall now be concerned with a particular class of spherical varieties: the wonderful varieties.

Wonderful $G$-varieties are classified by combinatorial objects supported on the Dynkin diagram of $G$ called spherical systems.
The purpose of this section is to establish an existence criterion of $\sigma$-equivariant real structures as well as quantitative properties of real loci
 of wonderful varieties in terms of these invariants, the automorphism $\varepsilon_\sigma$ of the Dynkin diagram of $G$ as well as the Cartan index
of a canonical  $\sigma$-self-conjugate simple $G$-module.

\subsection{Basic material}
The canonical embedding of a $G/H$ with $H\subset G$ being wonderful can be intrinsically defined.
Specifically, by a theorem of~\cite{Lu96}, 
a smooth complete $G$-variety $X$ is a smooth canonical embedding of a spherical homogeneous space if and only if
\begin{enumerate}
 \item $X$ contains an open $G$-orbit $X_G^\circ$;
 \item the complement $X\setminus X_G^\circ$ consists of a finite union of prime divisors $D_1,\ldots, D_r$ with normal crossings;
\item two points of $X$ are on the same $G$-orbit if (and only if) they are contained in the same $D_i$'s.
\end{enumerate}

We call a smooth complete $G$-variety \emph{wonderful of rank $r$} if it satisfies the aforementioned properties (1), (2) and (3).
 
As mentioned above, wonderful subgroups of $G$ (and in turn wonderful $G$-varieties) can be classified by more convenient invariants than the Luna-Vust invariants.
Let us recall how they are defined by Luna~\cite{Lu01}. One may consult also~\cite{T} for a survey.

Let $X$ be a wonderful $G$-variety.
Equivalently, consider a wonderful subgroup $H$ of $G$ and denote as previously its Luna-Vust invariants by $\mathcal X,V, \mathcal V, \mathcal D$.
The cone $\mathcal V$ being strictly convex and simplicial, it can be defined by inequalities.
More precisely, there exists a set $\Sigma_X$ of linearly independent primitive elements  such that
$$
\mathcal V=\left\{ v\in V:\, v(\gamma)\leq 0, \,\forall \gamma\in\Sigma_X\right\}.
$$
The set $\Sigma_X$ is called \emph{the set of spherical roots of $X$} (or $G/H$); it forms a basis of $V$ and, in turn, it also determines $\mathcal X$ entirely.
Consider now the set of colors $\mathcal D$.
Let
$$
P_X=\bigcap_{D\in\mathcal D} G_D.
$$
Obviously, $P_X$ is a parabolic subgroup of $G$ containing $B$; let thus $S_X$ be the set of simple roots associated
to $P_X$.

Finally, the third datum $\mathbf A_X$ attached to $X$ is a subset of $V$.
Given $\alpha\in \Sigma_X\cap S$, let 
$$
\mathbf A_X(\alpha)=\{\rho_D: D\in\mathcal D\mbox{ and } P_\alpha\cdot D\neq D\}\subset V
$$
where $B\subset P_\alpha$ stands for the parabolic subgroup of $G$ associated to $\alpha$.
Recall that the $\rho_D$ may not be distinct; we thus regard the set $\mathbf A_X(\alpha)$ as a multi-set.
The set $\mathbf A_X$ is defined as the union of the $\mathbf A_X(\alpha)$'s with $\alpha\in \Sigma_X\cap S$.

The triple $(S_X,\Sigma_X,\mathbf A_X)$ is called \emph{the spherical system of $X$} (or $G/H$).
We denote the spherical system of $X$ by $\mathscr S_X$ or, shortly, by $\mathscr S$.

Wonderful $G$-varieties are uniquely determined (up to $G$-isomorphism) by their spherical systems; see~\cite{Lo,CF}.

\subsection{An existence criterion for real structures}

Given $\sigma$, recall the definition of the associated automorphism $\varepsilon_\sigma$ of $S$ as well as its properties stated in Section~\ref{recallsauto}.

The spherical system $\mathscr S=(S_X,\Sigma_X,\mathbf A_X)$ of a wonderful $G$-variety $X$ is called \emph{$\varepsilon_\sigma$-stable} if 
the sets $S_X,\Sigma_X$ and $\mathbf A_X$ are stable by $\varepsilon_\sigma$.

\begin{lemma}\label{sphericalsystemstability}
Let $H\subset G$ be wonderful with spherical system  $\mathscr S=(S^p,\Sigma,\mathbf A)$.
Then $\sigma(H)\subset G$ is wonderful and its spherical system is the triple
$$
\varepsilon_\sigma(\mathscr S):=(\varepsilon_\sigma(S^p),\varepsilon_\sigma(\Sigma),\varepsilon_\sigma(\mathbf A)).
$$
\end{lemma}

\begin{proof}
By Proposition~\ref{wonderfulness}, $\sigma(H)\subset G$ is wonderful.
The assertion on the spherical systems follows readily from Lemma~\ref{Akhiezerlemma} and the recalls made at the beginning of this section.
\end{proof}

A spherical system of $G$ is called \emph{spherically closed} if the corresponding subgroup of $G$ is spherically closed.
Analogously, a wonderful $G$-variety is called \emph{spherically closed} if its spherical system is spherically closed.

\begin{theorem}\label{criterionwithsphericalsystem}
Let $X$ be a spherically closed wonderful $G$-variety with spherical system $\mathscr S$.
There exists a $\sigma$-equivariant real structure on $X$
if and only if  $\mathscr S$ is $\varepsilon_\sigma$-stable.
\end{theorem}

\begin{proof}
Suppose $\varepsilon_\sigma(\mathscr S)=\mathscr S$.
Then by Proposition~\ref{conjugationcriterion}, we have: $\sigma(H)=aHa^{-1}$ for some $a\in G$ and, in turn,
the mapping 
$$
\mu:G/H\rightarrow G/H,\quad gH\mapsto \sigma(g)aH
$$ is well-defined; the involutive property as well as its uniqueness are given by Corollary~\ref{uniquenessspherclosed}.
Moreover, thanks to the uniqueness of the wonderful embedding, $\mu$ can be extended to the whole $X$; see e.g.~\cite{ACF} for details.

Note that a $\sigma$-equivariant real structure on $X$ yields in particular a $\sigma$-equivariant real structure on the open $G$-orbit of $X$.
The converse thus stems from Lemma~\ref{sphericalsystemstability} and Theorem 2.1 in~\cite{Ak2}.
\end{proof}

A wonderful $G$-variety $X$ is called \emph{primitive} if it is not the parabolic induction of a wonderful variety 
(see Subsection~\ref{sectionparabolicinduction} for recollection of this notion)
nor the fiber product of wonderful varieties, meaning that $X$ is not $G$-isomorphic
to a wonderful $G$-variety equals to $X_1\times_{X_3} X_2$ with $X_i$ ($i=1,2,3$) being a wonderful $G$-variety.
By analogy, we call a wonderful subgroup $H\subset G$ \emph{primitive} if its canonical embedding is primitive.

 \begin{theorem}\label{conjugateprimitive}
Let $H\subset G$ be a primitive wonderful subgroup of $G$. 
If none of the spherical roots of $G/H$ is a simple root of $G$ then the subgroups $H$ and 
$\sigma(H)$ of $G$ are conjugate as soon as $(G, H, \sigma)$ is not one of the following triples.
\begin{enumerate}
\item $(SO_{4n},N_G(GL_{2n}), \sigma)$;
\item $(SO_{8},\mathrm{Spin}_7, \sigma)$;
\item $(SO_8, SL_2\cdot Sp_4,\sigma)$
\end{enumerate}
where $\sigma$ in (1), (2) and (3) defines the real form $SO_{p,q}$ with $p\leq q$ and $p,q$ odd.
\end{theorem}

\begin{proof}
The wonderful subgroups $H\subset G$ satisfying the assumptions of the theorem, together with their spherical systems, are listed in~\cite{BCF}.
This enables us to apply the criterion stated in Theorem~\ref{criterionwithsphericalsystem}.
Recall the definition of the automorphism $\varepsilon_\sigma$ stated in Section~\ref{recallsauto}; for convenience,
one may also consult Table 5 in~\cite{O} where $\varepsilon_\sigma$ together with the Satake diagrams are given.
We are thus left to check case-by-case which spherical systems under consideration are $\varepsilon_\sigma$-stable for a given $\sigma$.

We end up with the spherical systems numbered as $(34),(36)$ and $(37)$ in~\cite{BCF}, that is, with the groups
$(SO_8/SL_2\cdot Sp_4)$,$(SO_{8},\mathrm{Spin}_7)$ and $(SO_{4n},N_G(GL_{2n}))$.
Indeed, their spherical roots are 
$\{2\alpha_1,2\alpha_2,\alpha_3+\alpha_4\}$, $\{2\alpha_1+2 \alpha_2+\alpha_3+\alpha_4\}$ and 
$\{\alpha_1+2 \alpha_2+\alpha_3, \ldots, \alpha_{n-3}+2 \alpha_{n-2}+\alpha_{n-1}, 2\alpha_n\}$ respectively.
Note that these sets are not $\varepsilon_\sigma$-stable if $\sigma$ is the involution stated in the theorem.
\end{proof}

\begin{remark}\label{remark-strictw'ful}
Wonderful $G$-varieties whose points have a self-norma\-lizing stabilizer satisfy the assumption on the spherical roots made in the above theorem, that is,
none of their spherical roots is a simple root of $G$; see \cite{BCF} for details.
\end{remark}

\begin{corollary}
Let $G$ be a simple group and $X$ be an affine homogeneous spherical $G$-variety with weight monoid $\Gamma$.
Let $d$ denote the codimension of a generic orbit of the identity component of $G$ on $X$.
If $X$ can be equipped with a $\sigma$-equivariant real structure then $d=\mathrm{rk}\,\Gamma$.
\end{corollary}

\begin{proof}
Suppose that $d\neq\mathrm{rk}\,\Gamma$.
Write $X=G/H$. The triples $(G,H,\sigma)$ are given in Theorem 6.4 in~\cite{Ak}.
One thus observes that $(G,N_G(H),\sigma)$ are exactly the  triples stated in Theorem~\ref{conjugateprimitive}.
Therefore, for such triples, the subgroups $N_G(H)$ and $N_G(\sigma(H))=\sigma(N_G(H))$ of $G$ are not conjugate hence neither are 
the subgroups $H$ and $\sigma(H)$ of $G$.
We conclude the proof of the corollary by invoking Theorem~\ref{criterionwithsphericalsystem}.
\end{proof}

The importance of the assumptions we made in the previous statements is reflected in the following examples.

\begin{example}
Let $G=SL_{n+1}$ with ($n\geq 2$) and $P\subset G$ be the standard parabolic subgroup associated to the simple roots $\alpha_1$ and $\alpha_2$.
Consider the  variety $X=G\times_P X'$ with $X'$ being the $SL_3$-variety $\mathbb P^2\times(\mathbb P^2)^*$.
The varieties $X$ and $X'$ have a single spherical root, namely the root $\alpha_1+\alpha_2$.
Let $n$ be odd and $\sigma$ define the quasi-split but non-split real form of $G$.
If $n>2$  then the spherical system of $X$ is not $\varepsilon_\sigma$-stable since $\varepsilon_\sigma(\alpha_1+\alpha_2)=\alpha_{n-1}+\alpha_n$.
Note that here $\sigma(P)$ and $P$ are not conjugated subgroups of $G$.
\end{example}

\begin{example}
Let $G=SL_4$ with $\sigma$ defining the quasi-split but non-split real form of $G$.
Consider the spherical system of $G$ given by the triple $(\emptyset, \{\alpha_1,\alpha_2+\alpha_3\},\emptyset)$. 
This spherical system is not $\varepsilon_\sigma$-stable since $\varepsilon_\sigma(\alpha_1)=\alpha_3$.
The associated wonderful $G$-variety equals $X=X_1\times_{G/P} X_2$ where
$P$ is the standard parabolic subgroup of $G$ associated to the simple root $\alpha_3$,
$X_1$ (resp. $X_2$) is the parabolic induction of $\mathbb P^1\times\mathbb P^1$ (resp. $\mathbb P^2\times(\mathbb P^2)^*$) from the parabolic subgroup of $G$
with Levi subgroup of semisimple part $SL_2(\alpha_1)$ (resp. $SL_3(\alpha_2,\alpha_3)$).
\end{example}

\begin{remark}
A spherical subgroup $H$ of $G$ whose spherical closure has a $\varepsilon_\sigma$-stable spherical system
may not be conjugate to $\sigma(H)$, as the following example shows.
\end{remark}

\begin{example}
Let $G=SL_3$ with $\sigma$ defining the compact real form of $G$.
Consider the standard $G$-module $\mathbb C^3$ equipped its canonical basis $\{e_1, e_2, e_3\}$.
Let $V$ (resp. $H$) denote the line generated by (resp. the stabilizer in $G$ of) $e_3$.
The fiber bundle  $X=G\times_H V$ is thus a spherical affine $G$-variety
whose spherical system $(\emptyset, \{\alpha_1+\alpha_2\},\emptyset)$ is $\varepsilon_\sigma$-stable.
The coordinate ring of $X$ equals $\oplus_\Gamma V(\lambda)$ where $\Gamma=\mathbb N (\omega_1+\omega_2)+\mathbb N \omega_2$.
Since the weight $\omega_2$ is mapped to $\omega_1$ by $\varepsilon_\sigma$, the monoids $\Gamma$ and $\sigma(\Gamma)$ are distinct.
It follows from Lemma~\ref{Akhiezerlemma} along with Losev's theorem recalled in Subsection~\ref{anysphericalgp}	
that the generic stabilizer $H_0$ of the variety $X$ is not conjugate to $\sigma(H_0)$.
\end{example}

\subsection{Quantitative properties of real structures}

Throughout this subsection, $X$ denotes  a wonderful $G$-variety equipped with a $\sigma$-equiva\-riant real structure $\mu_\sigma$ (also denoted simply by $\mu$ when
no confusion can arise).
We also assume that $X$ is \emph{strict}, namely that all points of $X$ have a self-normalizing stabilizer in $G$.

We keep the notation and the terminology set up in Section~\ref{recall-color}; furthermore,
by $D_X$ we denote the divisor of $X$ equal to the sum of the colors of $X$ and by $\omega_X$ the $B$-weight associated to $D_X$.

By considering the simple $G$-module of highest weight $\omega_X$ as a submodule of the module of global sections $H^0(X,\mathcal O_X(D_X))$, Pezzini
gets the following important feature of strict wonderful varieties.

\begin{proposition}~\label{recall-strict-ppties}
There exists a unique $G$-equivariant embedding of $X$ in the projectivization of the simple $G$-module of highest
weight $\omega_X$.
\end{proposition}

\begin{proof}
The divisor $D_X$ being ample (thanks to~\cite{B1}), the proposition is a consequence of Theorem 5 in~\cite{Pe}.
\end{proof}

\subsection{}
Recall the definition of the complex conjugate module $V^\sigma$ of a given simple $G$-module $V$ as well as its Cartan index; see Section~\ref{recallsauto}.
The following statement is the generalization of Proposition~5.1 in~\cite{ACF} to an arbitrary involution $\sigma$.

\begin{proposition}\label{proposition-embbeding}
The simple $G$-module $V_X$ of highest weight $\omega_X$ sa\-tisfies the following properties.

\begin{enumerate}
 \item 
The $G$-modules $V_X$ and $V_X^\sigma$ are isomorphic.
\item
There exist an  anti-linear automorphism $\nu: V_X\rightarrow V_X$ and a $G$-equivariant embedding $\varphi: X\rightarrow \mathbb P(V_X)$
 such that
\begin{enumerate}
\item 
$\nu(\rho(g)v)=\rho(\sigma(g)v)$ for every $(g,v)\in G\times V_X$.
\item
$\varphi\circ\mu=\nu\circ\varphi$.
\end{enumerate}
\end{enumerate}
\end{proposition}

\begin{proof}
By assumption, $X$ has a $\sigma$-equivariant real structure. In particular, the set of colors of $X$ is $\varepsilon_\sigma$-stable
 and, in turn, $\varepsilon_\sigma (\omega_X)=\omega_X$. Thanks to Theorem~\ref{autoAkhiezer}, we obtain (1).
The second assertion can be derived from (1) along with Proposition~\ref{recall-strict-ppties}.
\end{proof}

\begin{remark}
Invoking the $\varepsilon_\sigma$ stability of the set of colors of $X$ as above, one can prove that $\mu(D_X)$ is rationally equivalent to $D_X$.
\end{remark}

\begin{corollary}\label{criterion-upto-Cartan-index}
If $X^\mu\neq\emptyset$ then the Cartan index of the simple $G$-module of highest weight $\omega_X$ is positive.
\end{corollary}

\begin{proof}
Since $X^\mu\neq\emptyset$, there exists $x\in X$ such that $\mu(x)=x$ and in turn, $\nu([v])=[v]$ with $\varphi(x)=[v]$ thanks to Proposition~\ref{proposition-embbeding}.
In particular $\nu(v)=a v$ for some $a\in\mathbb C^\times$ and $\nu^2(v)=a\bar{a}v$.
The assertion of the proposition thus follows from the definition of the Cartan index.
\end{proof}

\begin{remark}
Let $P$ be the parabolic subgroup of $G$ containing $B$
and such that $G/P$ is isomorphic to the closed $G$-orbit of $X$.
By \cite{BLV}, there exists a quasi-affine subvariety $Z_0$ of $X$ such that the open $B$-orbit of $X$ is isomorphic to $P^u\times Z_0$
where $P^u$ denotes the unipotent radical of $P$.
In~\cite{ACF}, it is proved that $Z_0$  contains real points w.r.t. $\mu_\sigma$ with $\sigma$ defining the real form of $G$.
Based on this statement as well as on Proposition~5.1 in loc. cit., the authors give an upper bound of the $G^\sigma$-orbits of $X^{\mu\sigma}$.
The author of the present article does not know if the aforementioned property on $Z_0$ generalizes properly to any involution $\sigma$, that is, 
if $Z_0$ contains real points w.r.t. $\mu_\sigma$ whenever 
so does $X$.
\end{remark}

From the definition of wonderful varieties, the $G$-orbits of $X$ are indexed by the subsets of $\{1,\ldots,r\}$ or equivalently by the subsets of $\Sigma$.
see Section~\ref{recallsauto}.
Further, given $I\subset\{1,\ldots,r\}$, the closure of the corresponding $G$-orbit within $X$ is a wonderful $G$-variety $X_I$.
Specifically, we have
$$
X_I=\bigcap_{i\in I} D_i
$$
and the spherical system of $X_I$ is $(S^p,\Sigma_I, \mathbf A_I)$
where 
$$
\Sigma_I=\{\gamma_i\in\Sigma: i\not\in I\}
$$ 
and $\mathbf A_I$ stands for the union of the $\mathbf A(\alpha)$'s such that $\alpha\in \Sigma_I$; see e.g Subsection 1.2 in~\cite{BL} for details.
Moreover, $X_I$ is obtained by parabolic induction from the parabolic subgroup $P_I$ of $G$
containing $B^-$ and associated to the set of  simple roots 
$$
S_I=S^p\cup \mathrm{Supp}\Sigma_I.
$$
Here $ \mathrm{Supp}\Sigma_I$ denotes the support of $\Sigma_I$, that is the subset of $S$ defined by the $\alpha$'s such that there exists $\gamma\in\Sigma_I$ with
$\gamma=\sum_{\beta\in S} a_\beta \beta$ and $a_\alpha\neq 0$.

\begin{theorem}\label{Thm-reallocus}
Let $X$ be a spherically closed wonderful $G$-variety endowed with a $\sigma$-equivariant real structure $\mu$.
Let $r$ denote the rank of $X$ and $(S^p,\Sigma,\mathbf A)$ be its spherical system.
The real points of $X$ are located on its $G$-orbits $G\cdot x_I$ $(I\subset \{1,\ldots, r\})$ such that
\begin{enumerate}
 \item $\Sigma_I=\varepsilon_\sigma(\Sigma_I)$ and 
\item $S_0\subset S_I$.
\end{enumerate}
In particular, if $\sigma$ defines the compact real form of $G$ then the real points w.r.t. $\mu$ are located on the open $G$-orbit of $X$.
\end{theorem}

\begin{proof}
Given $I\subset\{1,\ldots,r\}$, consider the corresponding $G$-orbit.
Suppose this orbit has a real point w.r.t $\mu$ then obviously so does its closure $X_I$ within $X$.

As recalled, $X_I$ is a wonderful $G$-variety whose set of spherical roots equals $\Sigma_I$.
Therefore, this set has to be $\varepsilon_\sigma$-stable by Theorem~\ref{criterionwithsphericalsystem}.
This proves the assertion stated in (1).

To prove that condition (2) has to be satisfied, recall that $X_I$ is parabolically induced from $P_I$.
Further, $X_I$ is also spherically closed; see Section 2.4 in~\cite{BP}.
By the uniqueness statement (Proposition~\ref{uniquenessspherclosed}), the real structure of $X_I$ is that described in Lemma~\ref{parabolicinduction}.
We can thus apply Proposition~\ref{realpointparabolicinduction}.
In particular, $G/P_I$ has a real point w.r.t. the real structure $gP_I\mapsto \sigma(g)P_I$.
This together with the last assertion of Theorem~\ref{autoAkhiezer} implies (2).
The theorem follows.
\end{proof}

A converse of the above theorem reads as follows.

\begin{proposition}\label{pptyS}
Let $X$ be a wonderful $G$-variety endowed with a $\sigma$-equivariant real structure $\mu$.
If the set $S_X$ contains the set $S_0$ associated to $\sigma$ then every $\mu$-stable $G$-orbit of $X$ contains real points w.r.t. $\mu$.
\end{proposition}

\begin{proof}
Recall that the projective $G$-orbit of $X$ is isomorphic to $G/P_X^-$ where $B^-\subset P_X^-\subset G$ is associated to $S_X$.

Since there exists a $\sigma$-equivariant real structure on $X$ by assumption, the spherical system of $X$ is $\varepsilon_\sigma$-stable;
see Theorem~\ref{criterionwithsphericalsystem}.
Furthermore, thanks to the assumption made on $S_X$ along with Theorem~\ref{autoAkhiezer}-(iii), we have: $\sigma(P_X^-)=P_X^-$.
The parabolic subgroup $P_X^-\subset G$ being self-normalizing hence spherically closed,
the real structure of $X$ restricted onto $G/P_X^-$ is (up to an automorphism of $G/P_X^-$) the mapping $gP_X^-\mapsto \sigma(g)P_X^-$;
see Proposition~\ref{uniquenessspherclosed}.
Since, as noticed $\sigma(P_X^-)=P_X^-$, the base point $eP_X^-\in G/P_X^-$ is a real point w.r.t. this mapping.
It follows that the projective $G$-orbit of $X$ contains real points w.r.t. $\mu$.

The rest of the proof just mimics that of Theorem 3.10 in~\cite{ACF}.
\end{proof}

The following examples show that we may encounter various situations for the set of real points of wonderful varieties.

\begin{example}
Let $G$ be of type $\mathsf E_7$ and $H\subset G$ be the normalizer of the stabilizer of a nilpotent element in the adjoint orbit 
of weighted diagram $(0100001)$.
As proved in ~\cite{BCF}, the subgroup $H\subset G$ is spherically closed with set of colors identified with the set of fundamental weights of $G$.
In case $\sigma$ defines the real form of type $\mathsf{EVI}$, the Cartan index of the $G$-module of highest weight equal to the sum of all fundamental weights is $-1$ 
(see Table 5 in~\cite{O}). 
By Corollary~\ref{criterion-upto-Cartan-index}, $G/H$ thus has no real points w.r.t. $\mu_\sigma$.
\end{example}

\begin{example}
 Let $G$ be of type $\mathsf E_8$. Then $\varepsilon_\sigma$ is trivial for every $\sigma$ and every finite dimensional $G$-module is of Cartan index $1$; see e.g. Table 5 in~\cite{O}.
Consider the nilpotent orbit $\mathcal O\subset \mathfrak g:=\mathrm{Lie}\,G$ of weighted diagram $(00000010)$.
Let $H\subset G$ be the stabilizer of $[e]\in \mathbb P(\mathfrak g)$ where $e\in\mathcal O$.
As pointed out in Appendix B of~\cite{BCF}, $H$ is a spherically closed subgroup of $G$.
Since $\varepsilon_\sigma$ is trivial, the spherical system of $G/H$ is $\varepsilon_\sigma$-stable.
Thanks to Theorem~\ref{criterionwithsphericalsystem},  the mapping $\mu_\sigma$ defines a $\sigma$-equivariant real structure on $G/H$.
Further, $S_X=\{\alpha_2,\alpha_3,\alpha_4,\alpha_5\}$; see again~\cite{BCF}.
Therefore $S_X$ fulfills the property of Proposition~\ref{pptyS} whenever $\sigma$ defines the real form  $\mathsf{EVIII}$ or $\mathsf{EIX}$
and in turn $G/H$ has real points w.r.t. $\mu_\sigma$, for these involutions $\sigma$.
This is in accordance with Djokovic's tables (\cite{D}) stating, in particular, that  $\mathcal O\cap \mathfrak g_{\mathbb R}\neq\emptyset$ when
$\mathfrak g_{\mathbb R}$ denotes the real form $\mathsf{EVIII}$ or $\mathsf{EIX}$ of $\mathfrak g$. 
\end{example}

\begin{example}
 Let $G$ be of type $\mathsf E_6$ and let $H\subset G$ be the normalizer of the stabilizer of a nilpotent element in the adjoint orbit $\mathcal O\subset\mathfrak g$
of weighted diagram $(000100)$.
The orbit $\mathcal O$ is spherical and $H\subset G$ is spherically closed; its spherical system is given by the triple
$(\emptyset$, $\Sigma=\{\alpha_1+\alpha_6,\alpha_3+\alpha_5,\alpha_2+\alpha_4\},\emptyset)$; see~\cite{BCF} for details.
Observe that this spherical system is $\varepsilon_\sigma$-stable for every $\sigma$.
Since $S_X=\emptyset$, the assumption of Proposition~\ref{pptyS}  is not fulfilled either for the real forms $\mathsf{EIII, EIV}$ or for the compact form of $\mathsf E_6$.
On the other hand, by Djokovic's tables, we know  that $\mathcal O\cap \mathfrak g_{\mathbb R}=\emptyset$ for the aforementioned real forms of $\mathsf E_6$.
\end{example}

\begin{example}
Let $G=G_1\times G_1$ with $G_1$ being a simple group.
Then the $G$-variety $G_1\simeq G/\mathrm{diag}(G)$ is spherical.
If $G_1$ is adjoint then $G_1\simeq G/\mathrm{diag}(G)$ is even wonderful and  $S_X=\emptyset$, 
so $S_X$ does not always fulfill the condition of Proposition~\ref{pptyS}.
Equip $G$ with the involution $\sigma=(\sigma_1,\sigma_1)$ where $\sigma_1$ is any anti-holomorphic involution of $G_1$.
This case gives an example where there are always real points in $G/H$ w.r.t. $\sigma$, whatever $\sigma_1$ is.
\end{example}


\end{document}